\documentclass[twocolumn]{autart}    % Enable this line and disable the 
\usepackage{graphicx}      % include this line if your document contains figures
\usepackage{natbib}        % required for bibliography
\usepackage{xcolor}
\usepackage{amsmath,amssymb,mathtools}
\usepackage{soul}
\usepackage{algorithm}
\usepackage{algpseudocode}
\usepackage{tikz}
\usetikzlibrary{positioning}

\renewcommand{\Re}{\mathbb{R}}
\newcommand{\calA}{\mathcal{A}}
\newcommand{\nxn}{{n\times n}}
\newcommand{\Ne}{\mathbb{N}}
\newcommand{\calD}{\mathcal{D}}

\newcommand{\SPn}{\Re^\nxn_{\succ0}}

\newcommand{\calV}{\mathcal{V}}

\newcommand{\rhoquad}{\rho_{\mathrm{quad}}}

\newcommand{\Sph}{\mathbb{S}}
\newcommand{\lambdamin}{\lambda_{\mathrm{min}}}

\DeclareMathOperator*{\argmin}{arg\,min}

\newcommand{\kappacond}{\overline\kappa}
\newcommand{\lambdamax}{\lambda_{\mathrm{max}}}
\newcommand{\mustate}{\mu^{\mathrm{state}}}
\newcommand{\mumode}{\mu^{\mathrm{mode}}}
\newcommand{\munoise}{\mu^{\mathrm{noise}}}
\newcommand{\qerror}{Q^{\mathrm{error}}}

\newcommand{\Vtilde}{\tilde{V}}
\newcommand{\Vover}{\overline{V}}
\makeatletter
\renewcommand{\thealgorithm}{\arabic{algorithm}}
\makeatother

\begin{document}

\begin{frontmatter}

\title{Direct vs. Indirect Data-Driven Control: Case-study of Switched Systems Stability}

\thanks[footnoteinfo]{\textbf{Preprint version. Submitted to Automatica.}\\GB is an FNRS Postdoctoral Researcher.
RJ is an FNRS honorary Research Associate.
This project has received funding from the European Research Council (ERC) under the European Union's Horizon 2020 research and innovation programme under grant agreement No 864017 -- L2C.}

\author[First]{Alexis Vuille} 
\author[First]{Guillaume O.~Berger} 
\author[First]{Rapha\"el M.~Jungers}

\address[First]{UCLouvain, Louvain-la-Neuve, Belgium (e-mail: firstname.lastname@uclouvain.be).}


\begin{abstract}                          % Abstract of not more than 200
A central methodological question in data-driven control is whether to adopt a direct or indirect approach.
Direct methods infer a controller or certificate directly from data, while indirect methods first identify a model and then apply model-based control techniques.
Recent developments of the direct method have led to finite-sample guarantees for the data-driven stability analysis of switched linear systems under various settings.
However, for the indirect method, such guarantees remain largely elusive.
In this paper, we provide a novel framework for the stability analysis of switched linear systems from noisy state measurements, using the indirect approach.
Our framework comes with finite-sample guarantees on the convergence rate of the system.
For that, we combine generalization bounds from machine learning and system identification with quadratic Lyapunov analysis.
To enable comparison, we also extend existing direct methods to handle measurement noise beyond the bounded-noise case available in the literature.
Finally, we compare the two approaches through numerical experiments, revealing that the indirect approach can give tighter probabilistic guarantees as well as greater robustness to noise and outliers than the direct approach.
%These observations reveal limitations of current state-of-the-art methods and highlight opportunities for their improvement.
\end{abstract}

\end{frontmatter}

\section{Introduction}
In recent years, data-driven methods have emerged as a central paradigm for the analysis and control of cyber-physical systems \cite{nejati2023formal,coulson2019data,de2019formulas,fan2017dryvr}. % add refs given in comments below.
This development is driven by several converging trends. Indeed, many modern systems are proprietary, highly complex, or governed by poorly understood physical principles, making modeling from first principles impractical.
At the same time, advances in sensing technologies and the widespread deployment of inexpensive and accurate sensors have increased the availability of data, complemented by user feedback and public databases. In parallel, statistical learning theory has undergone major theoretical and algorithmic developments, providing rigorous tools for extracting guarantees from finite datasets \cite{shalevshwartz2014understanding}.

An important question in data-driven control is the distinction between \emph{indirect} and \emph{direct} approaches, which has attracted significant attention; see, e.g., \cite{dorfler2023directvsindirect}.
In the indirect approach, one first identifies a model from data and then performs analysis or controller synthesis on the identified model. In contrast, the direct approach bypasses system identification and constructs certificates, such as Lyapunov or barrier functions, directly from data. While both paradigms are well established, their relative advantages remain insufficiently understood.%, in particular when one seeks formal guarantees from finite and noisy datasets. 

Recent works have started to investigate this question in the context of linear control problems; see, e.g., \cite{dorfler2023certaintyequivalence,krishnan2021direct,dorfler2023bridgingdirectvsindirect}. Although these studies provide valuable insights into the relative merits of direct and indirect methods, it remains unclear to what extent their conclusions extend beyond the linear setting. This motivates the search for more complex classes of systems that can serve as testbeds for investigating the direct-vs-indirect question.

In this paper, we consider this question in the context of switched linear systems. Switched linear systems constitute a paradigmatic class of hybrid systems that are more complex than linear systems and arise naturally in a wide range of applications or as approximation of nonlinear dynamics; see, e.g., \cite{liberzon2003switching,jungers2009thejoint,sun2011stability}.
Their stability is characterized by the joint spectral radius (JSR, \cite{rota1960note}), which captures the maximal exponential growth rate of trajectories of the system.
It is well known that computing the JSR exactly is a hard problem \cite{blondel2000boundedness,jungers2009thejoint}.
Nevertheless, several powerful approximation techniques have been developed in the model-based setting, including combinatorial and sum-of-squares relaxations \cite{ahmadi2014joint,jungers2009thejoint}.

{When a model of the system is not available, \cite{de2019formulas,bisoffi2022data} show how stability can be analyzed and stabilizing controllers can be designed for linear systems directly from trajectory data.
Several recent works also study the case of switched linear systems \cite{kenanian2019data,berger2021chanceconstrained,rubbens2021datadriven,wang2021datadriven,banse2024learning,vuille2024data,vuille2025stochastic}.
These approaches belong to the family of \emph{direct} methods: they construct Lyapunov functions (e.g., quadratic or polynomial) directly from data (which can include full or partial state and noisy measurements) and rely on tools from scenario optimization and statistical learning theory to derive probabilistic guarantees on the stability of the system.
Despite these advances, existing direct methods still face limitations in terms of robustness to noise and outliers \cite{banse2024learning}.
Furthermore, the literature has almost exclusively focused on the direct paradigm; comparatively, little attention has been devoted to the development of indirect approaches with finite-sample guarantees.

\textbf{Contributions. }
Our contribution is threefold.
First, we study the indirect approach for data-driven stability analysis of switched linear systems. Our proposed method first identifies a switched linear model from data (notably, any hybrid system identification method can be used, e.g., \cite{paoletti2007identification,lauer2019hybrid}); then, it upper bounds the JSR of the system by synthesizing a robust quadratic Lyapunov function for the identified model.
Crucially, we derive finite-sample probabilistic guarantees on this JSR upper bound by i) bounding the identification error with high probability and ii) propagating this bound to the JSR upper bound via sensitivity analysis of the quadratic Lyapunov function. For the first part, we leverage split-set methods from statistical learning theory (see \textit{Related work} for more information).
Our framework applies to systems with noisy full-state measurements.

Second, we revisit existing direct methods that account for noisy data, namely \cite{banse2024learning}, by i) refining the finite-sample guarantees in the presence of noisy measurements and ii) extending them to unbounded noise. Besides advancing the state of the art in direct data-driven control, this also enables fairer comparison with the indirect approach, which in our framework natively accounts for unbounded noise.

Finally, we compare the direct and indirect approaches on a numerical benchmark.
Our goal is to identify which approach is favorable in different settings, both in computation time and sample complexity.
Our results show that in the low-noise regime both approaches are able to provide stability guarantees for the benchmark system, even though the indirect approach requires  less data than the direct approach.
In the higher-noise regimes, only the indirect approach was able to provide meaningful stability guarantees, demonstrating greater robustness to noise and outliers.
The indirect approach also exhibits better computational complexity.
These findings are valuable (i) for practitioners who have to select one method for their application}, and (ii) to contribute to the overall debate on direct vs indirect methods by highlighting the strengths and weaknesses of both approaches on a representative benchmark.

\subsection*{Related work}

In the context of switched linear systems, several works have investigated stability certification from trajectory data.
These works follow the direct paradigm, learning a stability certificate directly from data and deriving probabilistic guarantees on its validity. The original result in \cite{kenanian2019data} shows how statistical guarantees derived from scenario optimization theory \cite{campi2008exact,calafiore2010random} can be translated into stability guarantees under a given confidence level. Another inspiration for the present work is \cite{banse2024learning}, which derives probabilistic upper bounds on the joint spectral radius under bounded measurement noise.
We extend the latter to unbounded measurement noise and improve it by exploiting symmetries of the problem.

In parallel, the identification of switched and hybrid systems from data has been extensively studied, leading to a broad literature
\cite{bako2011identification,paoletti2007identification,vidal2008recursive,lauer2019hybrid}. However, finite-sample guarantees remain scarce. A notable exception is \cite{massucci2022statistical}, which provides statistical guarantees for hybrid system identification.
Our framework differs from \cite{massucci2022statistical} in that our objective is not identification itself, but rather the use of identification guarantees for stability certification. This requires controlling the tail of the identification error, rather than its expectation as in \cite{massucci2022statistical}, in order to obtain sharp probabilistic stability guarantees. To this end, we leverage split-set methods from statistical learning theory (see, e.g., \cite{raschka2018model}), which provide simple and effective finite-sample confidence bounds that are particularly well suited to our setting. Split-set methods have already been used in control settings (for instance, conformal prediction in safety analysis \cite{fan2020statistical} and the holdout method in reachability analysis \cite{dietrich2025data}); however, their use in switched system identification has not been investigated yet to the best of our knowledge.

An important contribution of our work is an extensive comparison of the direct and indirect approaches for switched linear systems stability analysis.
Similar comparisons have been conducted for LTI system control in \cite{krishnan2021direct,dorfler2023bridgingdirectvsindirect,dorfler2023certaintyequivalence}, while a broader perspective on this methodological question is provided in \cite{dorfler2023directvsindirect}. These works show that, for LTI systems, direct and indirect approaches can be related through regularization or relaxation techniques, which in some cases lead to equivalent formulations, and that they exhibit trade-offs in data efficiency, asymptotic performance, and robustness.
Our work differs from these studies in that we focus on switched linear systems, which introduces additional challenges both for system identification and for stability certification.
We show through numerical experiments that the current state of the art provides direct and indirect techniques that are not equivalent in general, each with respective strengths and weaknesses.

\subsection*{Notation}
$\SPn$ denotes the set of $n\times n$ symmetric positive definite matrices.
For $N\in\Ne$, $[N]$ denotes the set $\{1,\ldots,N\}$. Given $x \in \mathbb{R}^n$, $\|x\|_2\coloneqq\sqrt{x^{\top}x}$ denotes its 2-norm. We denote the unit 2-norm sphere in $\Re^n$ by $\Sph^{n-1}$.

\section{Problem Statement}\label{sec:problem_statement}

We consider a discrete-time \emph{switched linear system} (SLS) whose dynamics is described by
\begin{equation}\label{eq:sys}
\xi(t+1) \in \{A\xi(t) : A\in\calA\},
\end{equation}
wherein $\calA=\{A_1,\ldots,A_m\}\subseteq\Re^\nxn$ is a set of $m$ matrices in $\Re^\nxn$.
Each matrix of $\calA$ is called a \emph{mode} of the system.
A \emph{trajectory} of this system is a function $\xi:\Ne\to\Re^n$ such that for all $t\in\Ne$, the condition in \eqref{eq:sys} holds.
Since $\calA$ fully characterizes system \eqref{eq:sys}, in the following, we will refer to it simply by $\calA$.
%Unless said otherwise, the system is always in dimension $n$ and with $m$ modes.

We say that $\calA$ is \emph{asymptotically stable} if all trajectories converge to the origin.
The rate of growth of the trajectories is called the joint spectral radius (JSR) of $\calA$:

\begin{defn}\label{def:jsr-rate-stability}
The \emph{joint spectral radius} of $\calA$, denoted by $\rho(\calA)$, is the infimum of all $r\geq0$ for which there exists $C\geq1$ such that every trajectory $\xi$ of $\calA$ satisfies that for all $t\in\Ne$, $\lVert\xi(t)\rVert\leq Cr^t\lVert\xi(0)\rVert$.
\end{defn}

The JSR is notoriously difficult to approximate, even when the matrices in $\calA$ are known \cite{jungers2009thejoint}.
One way to upper bound it is to compute the contraction rate of the system with respect to a given norm:

\begin{defn}\label{def:quad-contraction}
Given a SLS $\calA$ and a norm $\lVert\cdot\rVert:\Re^n\to\Re_{\geq0}$.
The \emph{contraction rate} of $\calA$ with respect to $\lVert\cdot\rVert$ is defined by
\[
\gamma(\lVert\cdot\rVert,\calA)=\max_{x\in\Re^n\setminus\{0\},\,A \in\calA}\frac{\lVert Ax\rVert}{\lVert x\rVert}.
\]\vskip0pt
\end{defn}
The following result is standard in the literature (see, e.g., \cite[Proposition 1.4]{jungers2009thejoint}):\footnote{In fact, it holds that $\rho(\calA)=\inf_{\lVert\cdot\rVert}\gamma(\lVert\cdot\rVert,\calA)$, where the infimum is over all norms on $\Re^n$ \cite[Proposition 1.4]{jungers2009thejoint}.}

\begin{prop}\label{prop:upper-bound}
Given a SLS $\calA$ and a norm $\lVert\cdot\rVert:\Re^n\to\Re_{\geq0}$, it holds that $\rho(\calA)\leq\gamma(\lVert\cdot\rVert,\calA)$.
\end{prop}

\begin{rem}\label{rem:lyap}
When $\gamma(\lVert\cdot\rVert,\calA) < 1$, the function $x\mapsto\lVert x\rVert$ is a Lyapunov function for the system.
%In general, it can be understood as a Lyapunov function for the system $\calA$ with each matrix divided by $\gamma(\lVert\cdot\rVert,\calA)$.
\end{rem}

In this paper, following previous work on data-driven control of SLS (e.g., \cite{kenanian2019data,berger2021chanceconstrained,wang2021datadriven,banse2024learning}), we focus on norms induced by a quadratic function, called hereafter \emph{quadratic norms}.
Namely, given a positive definite matrix $P$, we define the associated norm by $\lVert x\rVert_P\coloneqq(x^\top Px)^{1/2}$.
For simplicity of notation, we write $\gamma(P,\calA)=\gamma(\lVert\cdot\rVert_P,\calA)$.

We define the \textit{best quadratic approximation} of the JSR:
\[
\rhoquad(\calA) = \inf_{P\in\SPn} \gamma(P,\calA).
\]
To ensure existence of a minimizer and guarantee that it is well conditioned, we restrict the set to \emph{$\delta$-conditioned} positive definite matrices:
\begin{equation}\label{prob:whitebox-problem}
    \rhoquad(\calA,\delta) = \inf_{\substack{P\in\SPn\\I\preceq P\preceq \delta I}} \gamma(P,\calA),
\end{equation}
where $\delta \geq 1$ is the conditioning parameter.
We denote by $P_0(\calA,\delta)$ the matrix realizing the minimum in \eqref{prob:whitebox-problem}.\footnote{\label{fn:exis_unique}By compactness of the domain of $\delta$-conditioned matrices, the minimizer always exists.
If it is not unique, we can use a tie-breaking rule; see, e.g., \cite{berger2021chanceconstrained}.}

By Proposition~\ref{prop:upper-bound}, it holds that $\rho(\calA) \leq \rhoquad(\calA,\delta)$; see also \cite[Theorem 5]{kenanian2019data}.

\begin{rem}
Note that $\rhoquad$ can be computed efficiently using semi-definite programming (see, e.g., \cite{jungers2009thejoint}).
\end{rem}
%Indeed, for $\calA$, $\gamma\geq0$, and $P\in\SPn$, the inequality $\rho(\calA,P)\leq\gamma$ can be reformulated as the set of linear matrix inequalities (LMIs):
%\begin{equation}\label{eq:white-box-jsr}
%A^\top P A \preceq \gamma^2 P, \quad \forall\,A\in\calA.
%\end{equation}
%Consequently, $\rhoquad(\calA,\delta)$ can be approximated with accuracy $\epsilon>0$ efficiently---polynomial in time with respect to $n$, $\lvert\calA\rvert$ and $\log(\epsilon)$---as the problem can be formulated as a quasi-convex optimization problem.
%\end{rem}

\subsection{Data-driven analysis and random data collection}

In the data-driven setting, the system matrices are unknown so that the JSR cannot be computed directly as in \eqref{prob:whitebox-problem}. %--\eqref{eq:white-box-jsr}.
Instead, we rely on data consisting of noisy one-step trajectories.

More precisely, we consider an unknown SLS $\calA = \{A_1,\dots,A_m\}$.
An \emph{experiment} consists of initializing the system at a state $x \in \mathbb{S}^{n-1}$ and observing the state after one time step. Due to the switching and measurement noise, the observation satisfies
\[
y = A_j x + w,
\]
where $j \in [m]$ denotes the mode and $w \in \mathbb{R}^n$ the noise.

To formalize the data-generation process, we assume that each experiment produces a triplet $(x,j,w)$ drawn independently from an underlying probability distribution $\pi = \mustate \times \mumode \times \munoise$ on the set $\Delta \coloneqq \mathbb{S}^{n-1} \times [m] \times \mathcal{W}$.
The components of this product distribution have the following interpretation:
\begin{itemize}
\item $\mustate$ is the distribution of the initial condition $x$.
We assume that $x$ is sampled uniformly on $\mathbb{S}^{n-1}$.
\item $\mumode$ is the unknown distribution governing the switching mechanism. It models the probability with which each mode is activated.
\item $\munoise$ is the unknown noise distribution.
\end{itemize}\vskip0pt

\begin{rem}
The assumption on $\mustate$ can be interpreted in two ways:
(i) as a design choice, in an experimental setting where the experimenter actively initializes the system with uniformly distributed states (see for instance the adaptive-sampling framework in \cite{vuille2024data}); or
(ii) as a modeling assumption in a passive data-collection setting (such as in \cite{kenanian2019data,berger2021chanceconstrained,wang2021datadriven,banse2024learning}).
\end{rem}

With the assumptions above, repeating the experiment $N$ times yields an i.i.d. sample set
\begin{equation}\label{eq:sample_set_equation}
\omega_N = \{(x_i,j_i,w_i)\}_{i=1}^N.
\end{equation}
For each experiment $i\in[N]$, only the ``input--output'' pair $(x_i,y_i)$, where $y_i = A_{j_i}x_i + w_i$, is observable; the mode $j_i$ and the noise realization $w_i$ remain latent.
This gives the \emph{observation set}
\begin{equation}\label{eq:observation-set}
\calD\coloneqq \{(x_i,y_i)\}_{i=1}^N, \quad y_i = A_{j_i}x_i + w_i \quad \forall\,i\in[m].
\end{equation}
The following assumptions will be used to establish the probabilistic stability guarantees:

\begin{assum}[Mode excitation]\label{assump-mode-sampling}
There exists $\alpha>0$ such that each mode has probability at least $\alpha$ of being sampled, i.e., for all $j \in [m]$, $\mumode(j) \geq \alpha$.
\end{assum}

\begin{assum}[Noise tail bound]\label{assum:noise}
For all $W\geq0$, there exists a known bound $p(W)$ such that
\[
\munoise(\{ w \in \mathbb{R}^n : \|w\|_2 > W \}) \leq p(W).
\]\vskip0pt
\end{assum}
Note that Assumption~\ref{assum:noise} does not require full knowledge of $\munoise$ but only an upper bound on the tail probability of the noise. Two examples are given in Appendix~\ref{appendix:exmp}.

\subsection{A useful result}
We recall the following result from \cite[Theorem 15]{kenanian2019data}, stated here under a generalized form that will be useful to obtain probabilistic upper bounds. This result allows to translate any upper bound on the measure of the set of points where the quadratic norm fails to be contractive (Eq.~\ref{eq:violation-contraction} below) into an upper bound on the JSR:
% In particular, stability can be certified whenever this upper bound is smaller than one.}

\begin{prop}\label{prop:kenenian}
Consider a SLS $\calA$.
Let $\gamma\geq 0$, $P\in\SPn$ and $\nu\in(0,1)$.
For each $j\in[m]$, define 
\begin{equation}\label{eq:violation-contraction}
\mathcal{V}_j\coloneqq \left\{ x \in \mathbb{S}^{n-1} :\lVert A_j x\rVert_P> \gamma \lVert x\rVert_P\right\}.
\end{equation}
Assume that for all $j \in [m]$, $\mustate\big(\mathcal{V}_j\big) \leq \nu$.
Then,
\[
\rho(\calA)\leq \frac{\gamma}{\zeta\!\left(\frac\nu2\kappa(P)\right)},
\]
where 
\begin{itemize}
\item $\kappa(P) \coloneqq \sqrt{\det(P)/\lambda_{\mathrm{min}}(P)^n}$,
\item $\zeta(x) \coloneqq \sqrt{1-I^{-1}(2x;(n-1)/2;1/2)}$ for $x\in(0,1)$.
\item $I(\cdot;a;b)$ is the regularized incomplete beta function, defined by
$
I(x;a;b) \coloneqq \frac{\int_{0}^x t^{a-1} (1-t)^{b-1} \mathrm{d}t}{\int_{0}^1 t^{a-1} (1-t)^{b-1} \mathrm{d}t}.
$
\end{itemize}
\end{prop}

The proof is adapted from \cite[Theorem 15]{kenanian2019data}.
For completeness, it is given in Appendix~\ref{appendix:proof_kenenian}.

Key steps in the proofs of both the indirect and the direct approaches guarantees will be to bound the probability of $\calV_j$ in \eqref{eq:violation-contraction} in order to apply Proposition~\ref{prop:kenenian}.

\section{Indirect Approach}\label{sec:indirect_approach}

In this section, we present our first main contribution: an indirect-approach framework to bound the JSR with finite-sample probabilistic guarantees.
First, we explain how an identified SLS model along with a bound on the identification error can lead to a bound on the JSR of the true system (Theorem~\ref{thm:indirect-jsr-bound}); then, we explain how the identification error can be bounded by using tools from statistical learning (Proposition~\ref{prop:binomial-tail-bound}).

\subsection{From system identification to stability guarantees}\label{ssec:si-to-stability}

In the indirect approach, given an observation set $\calD$ as in \eqref{eq:observation-set}, we use system identification to learn a SLS model.
Then, we compute a quadratic norm and the associated contraction rate for this model. The main question is: since the contraction analysis is performed on the model, what guarantees can we obtain on the true system?
As we show in this section, we can transfer the guarantees from the model to the true system by bounding the discrepancy between the two dynamics.

In the following of this section, we let $\hat{\mathcal{A}} \coloneqq \{\hat{A}_1,\hdots,\hat{A}_m\}$ denote the identified SLS model.\footnote{For simplicity, we assume that the number of modes of the true and the identified systems are the same, but our framework can be easily extended to the more general setting.}

\begin{defn}\label{def:pred-err}
Given a SLS $\mathcal{A}$, a SLS model $\hat{\mathcal{A}}$ and a sample point $(x,j,w) \in \Delta$, we define the \emph{prediction error} as
\begin{equation*}  
e_{\mathcal{A},\hat{\mathcal{A}}}(x,j,w) \coloneqq \min_{k \in [m]} \|\hat{A}_k x-A_jx+w\|_2.
\end{equation*}
\end{defn}

The prediction error above is defined for a specific sample $(x,j,w)$.
We define the associated cumulative distribution of the prediction error as follows:

\begin{defn}\label{def:tail-pred-err}
Given a SLS $\mathcal{A}$, a SLS model $\hat{\mathcal{A}}$ and $\epsilon > 0$, the \emph{complementary cumulative distribution of the prediction error} is
\[ \qerror_{\mathcal{A},\hat{\mathcal{A}}}(\epsilon)
\coloneqq \pi\left(\left\{(x,j,w) \in \Delta : e_{\mathcal{A},\hat{\mathcal{A}}}(x,j,w) > \epsilon\right\}\right).
\]\vskip0pt
\end{defn}
The theorem below assumes an upper bound $\hat{q}$ on $\qerror_{\mathcal{A},\hat{\mathcal{A}}}(\epsilon)$ (we explain in Section \ref{ssec:si-bound} how to compute such an upper bound with finite-sample guarantees).
The theorem shows how this bound can be propagated into the quadratic norm computed for the model in order to obtain an upper bound on the JSR of the true system:

\begin{thm}\label{thm:indirect-jsr-bound}
Given a SLS $\mathcal{A}$, a model SLS $\hat{\mathcal{A}}$ and $\delta \geq 1$, let $P\coloneqq P_0(\hat{\mathcal{A}},\delta)$ and $\rhoquad(\hat{\mathcal{A}},\delta)$ be solutions of \eqref{prob:whitebox-problem}.
Let $\epsilon >0$ and $W\geq0$.
Let Assumptions \ref{assump-mode-sampling} and \ref{assum:noise} hold.\footnote{We remind that Assumptions \ref{assump-mode-sampling} and \ref{assum:noise} define the quantities $\alpha$ and $p(W)$ used in the definition of $\overline{\rho}(\hat{q};\hat{\mathcal{A}},\delta)$.}
Let $\hat{q}\geq\qerror_{\mathcal{A},\hat{\mathcal{A}}}(\epsilon)$.
The following %inequality 
holds:
\begin{equation*}
\rho(\mathcal{A}) \leq \underbrace{\frac{\rhoquad(\hat{\mathcal{A}},\delta)+(\epsilon+W)\sqrt{\kappacond(P)}}{\zeta\left(\frac{\hat{q}}{2(1-p(W))\alpha}\kappa(P)\right)}}_{\eqqcolon\:\overline{\rho}(\hat{q};\hat{\mathcal{A}},\delta)},
\end{equation*}
where
\begin{itemize}
\item $\kappacond(P) \coloneqq \lambdamax(P)/\lambdamin(P)$,
\item $\zeta(x)$ and $\kappa(P)$ are defined in Proposition \ref{prop:kenenian}.
\end{itemize}
\end{thm}

To prove Theorem \ref{thm:indirect-jsr-bound}, we will need the following lemma.
Intuitively, the lemma propagates the error coming from the noise and the system identification error into an upper bound on the measure of the set of points for which the contraction of the quadratic norm computed with the model does not hold in the true system.
This upper bound will be used to apply Proposition \ref{prop:kenenian} to obtain the upper bound on the JSR in Theorem \ref{thm:indirect-jsr-bound}.

\begin{lem}\label{lem:indirect-ub}
Let $\mathcal{A}$, $\hat{\mathcal{A}}$, $\delta$, $P$, $\rhoquad(\hat{\mathcal{A}},\delta)$, $\epsilon$, $W$ and $\hat{q}$ be as in Theorem \ref{thm:indirect-jsr-bound}.
Let Assumptions \ref{assump-mode-sampling} and \ref{assum:noise} hold.
Let $j\in[m]$ and $\calV_j$ be as in \eqref{eq:violation-contraction} with $\gamma \coloneqq \rhoquad(\hat{\mathcal{A}},\delta) + (\epsilon + W)\sqrt{\kappacond(P)}$.
Then,
\begin{equation*}
\mustate(\calV_j) \leq \frac{\hat{q}}{(1-p(W))\alpha}.
\end{equation*}
\end{lem}
\begin{pf}
Let
\[
\Omega_j \coloneqq \left\{x \in \mathbb{S}^{n-1}:\min_{k \in [m]} \|\hat{A}_kx-A_jx\|_2>\epsilon +W\right\}.
\]
We first prove the following bound:
\begin{equation}\label{eq:bound-omega}
\mustate(\Omega_j) \leq \frac{\hat{q}}{(1-p(W))\alpha}.
\end{equation}
To prove this, define
\begin{align*}
&\Omega\coloneqq \Big\{(x,j,w) \in \Delta : \min_{k \in [m]} \|\hat{A}_kx-A_jx\|_2>\epsilon+W\Big\}, \\
&\Vtilde\coloneqq \Big\{(x,j,w) \in \Delta : \min_{k \in [m]} \|\hat{A}_kx-A_jx+w\|_2>\epsilon\Big\}, \\
&\Vover\coloneqq \Big\{(x,j,w) \in \Delta : (x,j,w)\in\Omega \:\land\: \|w\|_2 \leq W\Big\}.
\end{align*}
Note that $\pi(\Vtilde)=\qerror_{\calA,\hat\calA}(\epsilon)$.
We prove that $\Vover \subseteq \Vtilde$, which implies $\pi(\Vover) \leq \pi(\Vtilde)$.
Therefore, let $(x,j,w) \in \Vover$.
By the triangle inequality,
\begin{align*}
\epsilon + W &< \min_{k \in [m]} \|\hat{A}_kx-A_jx\|_2 \\
%&=\min_{k \in [m]} \|\hat{A}_kx-A_jx+w-w\|_2 \\
&\leq \min_{k \in [m]} \|\hat{A}_kx-A_jx+w\|_2 + \|w\|_2 \\
&\leq \min_{k \in [m]} \|\hat{A}_kx-A_jx+w\|_2 + W.
\end{align*}
\noindent Hence, 
$
\epsilon < \min_{k \in [m]} \|\hat{A}_kx-A_jx+w\|_2,
$
so that $(x,j,w) \in \Vtilde$. This proves $\Vover\subseteq \Vtilde$.

Then, notice that
\begin{align*}
\pi(\Vover) &= \pi(\Omega) \cdot \munoise(\{w \in \mathbb{R}^n:\|w\|_2 \leq W\})\\
&\geq \pi(\Omega) \cdot (1-p(W)),
\end{align*}
where the latter comes from Assumption \ref{assum:noise}. This implies
\[
\pi(\Omega)\leq \frac{\pi(\Vover)}{1-p(W)} \leq \frac{\pi(\Vtilde)}{1-p(W)} = \frac{\qerror_{\mathcal{A},\hat{\mathcal{A}}}(\epsilon)}{1-p(W)} \leq  \frac{\hat{q}}{1-p(W)}.
\]
Since $\pi(\Omega) = \sum_{j=1}^m \mustate(\Omega_j)\mumode(j)$, it follows that
\[
\mustate(\Omega_j)\mumode(j) \leq \frac{\hat{q}}{1-p(W)}.
\]
Thus, by Assumption \ref{assump-mode-sampling},
\[
\mustate(\Omega_j) \leq \frac{\hat{q}}{(1-p(W))\mumode(j)}\leq \frac{\hat{q}}{(1-p(W))\alpha}.
\]
This proves \eqref{eq:bound-omega}. To prove the rest of the lemma, we prove that $\calV_j \subseteq \Omega_j$, or equivalently that 
\[
x \in (\mathbb{S}^{n-1} \setminus \Omega_j) \implies x \in (\mathbb{S}^{n-1} \setminus \calV_j). 
\]
It will follow that $\mustate(\calV_j) \leq \mustate(\Omega_j) \leq \frac{\hat{q}}{(1-p(W))\alpha}$.
Given $x \in (\mathbb{S}^{n-1} \setminus \Omega_j)$, let $k\in[m]$ be such that $\|A_jx-\hat{A}_kx\|_2\leq \epsilon+W$.
By the triangle inequality,
\begin{equation*}
        \|A_jx\|_P \leq \|\hat{A}_kx\|_P+\|A_jx-\hat{A}_kx\|_P
\end{equation*}
By the definition of $P$, it holds that
\[
\|\hat{A}_{k}x\|_P \leq \rhoquad(\hat\calA,\delta) \|x\|_P.
\]
Furthermore, it holds that
\begin{align*}
\|A_jx-\hat{A}_{k}x\|_P &\leq \sqrt{\lambdamax(P)} \|A_jx-\hat{A}_{k}x\|_2 \\
&\leq \sqrt{\lambdamax(P)} (\epsilon+W) \\
&\leq \sqrt{\frac{\lambdamax(P)}{\lambdamin(P)}} (\epsilon+W) \|x\|_P.
\end{align*}
By combining the two inequalities above, we obtain
\[
\|A_jx\|_P \leq \left(\rhoquad(\hat\calA,\delta) +\sqrt{\kappacond(P)}  (\epsilon+W)\right)\|x\|_P.
\]
Thus, $x\notin\calV_j$.
Hence, $\calV_j \subseteq \Omega_j$, concluding the proof.\qed
\end{pf}

{\renewcommand{\Elproofname}{Proof of Theorem \ref{thm:indirect-jsr-bound}.}
\begin{pf}
The proof follows by applying Proposition \ref{prop:kenenian} with $\nu \coloneqq\frac{\hat{q}}{(1-p(W))\alpha}$, which holds by Lemma \ref{lem:indirect-ub}.\qed
\end{pf}}
\subsection{Bounding the identification error}\label{ssec:si-bound}

Theorem \ref{thm:indirect-jsr-bound} allows us to obtain probabilistic bounds on the JSR through a bound on the complementary cumulative distribution of the prediction error, $\hat{q} \geq \qerror_{\mathcal{A},\hat{\mathcal{A}}}(\epsilon)$.
However, how do we compute such a bound on the distribution of the prediction error?
In this section, we propose a standard approach from statistical learning called the \textbf{binomial-tail inversion} to do this with finite-sample guarantees. This method, which belongs to the family of split-set (aka. holdout or validation-set) methods, splits the dataset into two parts: one for training and one for testing.
It then uses the binomial distribution that describes the number of samples in the testset for which the prediction error is larger than $\epsilon$. By inverting it, and comparing with the observed number of samples with prediction error larger than $\epsilon$, one can retrieve a probabilistic upper bound on $\qerror_{\mathcal{A},\hat{\mathcal{A}}}(\epsilon)$.

Concretely, binomial-tail inversion is a general framework to bound with prescribed confidence the probability of an event $S\subseteq\Omega$, where $(\Omega,\pi)$ is a probability space \cite{langford2005tutorial}.
It proceeds by taking $N$ i.i.d.~samples in $\Omega$ and computing how many of them belong to $S$.
In our case, given a SLS $\calA$, a SLS model $\hat\calA$ and $\epsilon>0$, the associated event is defined by
\begin{equation}\label{eq:event-set}
S\coloneqq\left\{(x,j,w) \in \Delta : e_{\mathcal{A},\hat{\mathcal{A}}}(x,j,w) > \epsilon\right\},
\end{equation}
i.e., it is the set of samples for which the prediction error is larger than $\epsilon$.
Note that $\pi(S)=\qerror_{\calA,\hat\calA}(\epsilon)$.
%\[
%\pi(S_{\hat\calA})=\qerror_{\calA,\hat\calA}(\epsilon).
%\]
For any sample set $\omega_N=\{(x_i,j_i,w_i)\}_{i=1}^N$, the number of samples in $S$ is denoted by
\begin{equation}\label{eq:N-unsat}
\begin{array}{rrl}
F(\omega_N) &\coloneqq& \sum_{i=1}^N \boldsymbol{1}_{(x_i,j_i,w_i)\in S} \\
&=& \#\{i\in[N] : e_{\mathcal{A},\hat{\mathcal{A}}}(x_i,j_i,w_i) > \epsilon\}.
\end{array}
\end{equation}
Binomial-tail inversion gives an upper bound on $\pi(S)$ from $F(\omega_N)$, that is correct with high probability:
%\footnote{Note that, by the law of large numbers, when $\omega_N$ is sampled i.i.d.~according to $\pi$, the ratio $\frac1N F_{\hat\calA}(\omega_N)$ gives a consistent estimator of $\pi(S_{\hat\calA})$. However, we need more than an \emph{asymptotic} guarantee; hence, the use of the binomial-tail inversion.}

\begin{prop}\label{prop:binomial-tail-bound}
Consider a SLS $\mathcal{A}$ and a SLS model $\hat{\calA}$. Let $\beta\in(0,1)$, $\epsilon>0$ and $F(\cdot)$ be as in \eqref{eq:N-unsat}.
For any $k\in\Ne$, $k\leq N$, let $\overline{q}^{\mathrm{bin}}_\beta(k)$ be the unique $q\in[0,1]$ such that $\Phi(q;k;N)=\beta$, where $\Phi$ is the cumulative distribution function of the binomial distribution defined as
\[
\Phi(q;k;N) = \sum_{j=0}^{k} \binom{N}{j}q^j(1-q)^{N-j}.
\]
The following inequality holds:
\[
\pi^N\left(\left\{\omega_N\in\Delta^N:\qerror_{\mathcal{A},\hat{\mathcal{A}}}(\epsilon) \leq \overline{q}^{\mathrm{bin}}_\beta\big(F(\omega_N)\big)\right\}\right) \geq 1-\beta.
\]\vskip0pt
\end{prop}

For a proof, see, e.g., \cite[Theorem 3.3]{langford2005tutorial}.

\subsection{Overall indirect framework}\label{ssec:overall_indirect}

The overall framework is presented in Algorithm \ref{alg:indirect}.
The soundness of the algorithm is established below:

\begin{thm}
Consider a SLS $\mathcal{A}$. Let $\beta\in(0,1)$ and $\epsilon>0$.
Let $\omega_N$ be an i.i.d.~sample set as in \eqref{eq:sample_set_equation} and $\calD$ be the observation set as in \eqref{eq:observation-set}. Let Assumptions \ref{assump-mode-sampling} and \ref{assum:noise} hold.
Consider an execution of Algorithm \ref{alg:indirect} with $\calD$, $\epsilon$ and $\beta$ as inputs and $\overline{\rho}$ as output.
With confidence $1-\beta$ on the sampling of $\omega_N$, it holds that $\rho(\mathcal{A}) \leq \overline{\rho}$.
\end{thm}
\begin{pf}
Corollary of Theorem \ref{thm:indirect-jsr-bound} and Proposition~\ref{prop:binomial-tail-bound}.\qed
\end{pf}
\begin{algorithm}[h]
\caption{Indirect data-driven stability certification}
\label{alg:indirect}
\begin{algorithmic}[1]

\Require Dataset $\calD\coloneqq\{(x_i,y_i)\}_{i=1}^N$ of noisy trajectories from unknown SLS $\calA$; system identification error level $\epsilon > 0$; confidence level $\beta\in(0,1)$.
\Ensure Probabilistic upper bound $\overline{\rho}$ on $\rho(\calA)$.

\State \textbf{System identification:}  
Learn a SLS model $\hat{\mathcal{A}} = \{\hat{A}_1, \dots, \hat{A}_m\}$ from $\mathcal{D}$.
\emph{Any approach can be used for that} (see also Section~\ref{sec:exp}).

\State \textbf{Bounding identification error:}  
Compute a probabilistic bound $\hat{q}$ such that
\[
\qerror_{\mathcal{A},\hat{\mathcal{A}}}(\epsilon) \leq \hat{q}
\]
holds with confidence $1-\beta$. \emph{See Proposition \ref{prop:binomial-tail-bound}}.

\State \textbf{Model-based quadratic JSR analysis:}  
Compute a matrix $P_0(\hat{\mathcal{A}},\delta)$ and the associated contraction rate $\rho_{\mathrm{quad}}(\hat{\mathcal{A}},\delta)$ for the model $\hat\calA$, as solutions of \eqref{prob:whitebox-problem}.

\State \textbf{Bound on the JSR of the system:}  
Compute $\overline{\rho} = \overline{\rho}(\hat{q}; \hat{\mathcal{A}}, \delta)$ as in Theorem~\ref{thm:indirect-jsr-bound}.
\end{algorithmic}
\end{algorithm}

\section{Direct Approach}\label{sec:direct_approach}

In this section, we recall the direct approach, developed in \cite{kenanian2019data,berger2021chanceconstrained,wang2021datadriven,banse2024learning}, with which we will compare our method in Section~\ref{sec:exp}.
We also extend and improve this framework to account for unbounded observation noise, to enable fair comparison with our approach.

The direct approach works as follows: from a sample set of (possibly noisy) trajectories, we compute a quadratic norm and its associated \emph{scenario} contraction rate by solving a quasi-convex scenario optimization program (see Eq.~\ref{prob:data-driven_problem} below).
Using tools from scenario optimization \cite{campi2008exact,calafiore2006thescenario}, this rate is then transformed into a probabilistic upper bound on the JSR of the system.

Concretely, given $\delta \geq 1$ and a sample set $\omega_N=\{(x_i,j_i,\allowbreak w_i)\}_{i=1}^N$, we solve the following scenario program:
\begin{equation}\label{prob:data-driven_problem}
\begin{array}{cl}
\min\limits_{\gamma\geq0,\,P\in\SPn} & \gamma,  \\
\text{s.t} & I \preceq P \preceq \delta I, \\[2pt]
& \|A_{j_i}x_i + w_i\|_P \leq \gamma \|x_i\|_P \quad \forall i\in[N],
\end{array}
\end{equation}
The optimal cost and solution of \eqref{prob:data-driven_problem} are denoted respectively by $\gamma_\star(\omega_N,\delta)$, and $P_\star(\omega_N,\delta)$.\footnote{For the same reason as mentioned earlier (footnote \ref{fn:exis_unique}), we can assume existence and uniqueness of the optimal solution.}They correspond to the scenario contraction rate and the quadratic norm.
\begin{rem}\label{rem:direct-dataset}
Note that while \eqref{prob:data-driven_problem} is defined with respect to the sample set $\omega_N$, it can be computed from the observation set $\calD$, defined in \eqref{eq:observation-set}.
Accordingly, one could write $\gamma_\star(\calD,\delta)$ and $P_\star(\calD,\delta)$ instead to highlight that it depends only on observable quantities.
\end{rem}
We make the following assumption on the noise distribution, which is also present in previous works on the direct approach with noisy measurements \cite{banse2024learning}:
\begin{assum}[Absolute continuity]\label{assum:zero-measure}
The noise distribution $\munoise$ is absolutely continuous with respect to the Lebesgue measure on $\mathbb{R}^n$.
\end{assum}

We prove the following theorem, which is an improvement (tighter bound) and a generalization (from bounded to unbounded noise) of \cite[Theorem 1]{banse2024learning}.
\begin{thm}\label{thm:jsr-direct-bound}
Consider a SLS 
$\mathcal{A}$.
Let $\beta\in(0,1)$ and $\delta\geq1$.
Let $d=n(n+1)/2$ and $N\in\Ne$, $N\geq d$.
For each sample set $\omega_N$ as in \eqref{eq:sample_set_equation}, let 
$\gamma_\star(\omega_N,\delta)$ and $P_\star(\omega_N,\delta)$ be the solution of \eqref{prob:data-driven_problem}.
Let $W\geq0$.
Let Assumptions~\ref{assump-mode-sampling},~\ref{assum:noise} and \ref{assum:zero-measure} hold.\footnote{We remind that Assumptions \ref{assump-mode-sampling} and \ref{assum:noise} define the quantities $\alpha$ and $p(W)$ used in the definition of $\overline{\rho}(\hat{q};\hat{\mathcal{A}},\delta)$.}
The following
holds:
\[
\pi^N\!\left(
\{\omega_N\in\Delta^N:\rho(\mathcal{A})\le\overline{\rho}(\omega_N)\}
\right)
\ge 1-\beta ,
\]
where 
\begin{itemize}
    \item the bound $\overline{\rho}(\omega_N)$ is defined by
    \[
    \overline{\rho}(\omega_N)
    \coloneqq
    \frac{\overline{\gamma}(\omega_N,\delta)}
    {\zeta\!\left(
    \frac{\epsilon(\beta,N)\kappa(P_\star(\omega_N,\delta))}
    {2\alpha\left(1-\frac12p(W)\right)}
    \right)} .
    \]
    \item $\overline{\gamma}(\omega_N,\delta)$ and $\epsilon(\beta,N)$ are defined in Lemma \ref{lem:direct-violation-set_bound} below.\\
    \item $\kappa(P)$ and $\zeta(x)$ are defined in Proposition ~\ref{prop:kenenian}.
\end{itemize}
\end{thm}\vskip0pt

\begin{rem}
We note that Theorem~\ref{thm:jsr-direct-bound} improves the bound in \cite[Theorem 1]{banse2024learning} even when the noise is bounded, say by $W_{\mathrm{max}}$.
Indeed, if $W_{\mathrm{max}}$ is large, choosing a smaller value $W < W_{\mathrm{max}}$ can yield a tighter upper bound $\overline{\rho}(\omega_N)$ (in some cases, even $W = 0$; see Subsection \ref{subsec:bounded_noise_exp}).
\end{rem}
\begin{rem}
Similarly to Remark~\ref{rem:direct-dataset}, $\overline{\rho}(\omega_N)$ can be written $\overline{\rho}(\calD)$ as it depends only on $\calD$, defined in \eqref{eq:observation-set}.
\end{rem}

To prove Theorem~\ref{thm:jsr-direct-bound}, we need the following lemma:
\begin{lem}\label{lem:direct-violation-set_bound}
Let $\calA$, $\beta$, $\delta$, $d$, $N$, $\gamma_\star(\omega_N,\delta)$, $P_\star(\omega_N,\delta)$ and $W$ be as in Theorem \ref{thm:jsr-direct-bound}.
Let Assumptions~\ref{assump-mode-sampling},~\ref{assum:noise} and \ref{assum:zero-measure} hold.
Define $\epsilon(\beta,N)$ as the unique $q\in[0,1]$ such that $\Phi(q;d-1;N)=\beta$, where $\Phi(\cdot;\cdot;\cdot)$ is as in Proposition \ref{prop:binomial-tail-bound}.
The following inequality holds:
\[
\pi^N\!\left(
\left\{\omega_N \in \Delta^N :
\pi(V(\omega_N)) \leq \xi(\beta,N)
\right\}
\right)
\geq 1-\beta,
\]
where $\xi(\beta,N) \coloneqq \frac{\epsilon(\beta,N)}{1-\frac12p(W)}$ and
\begin{align*}
&V(\omega_N) \coloneqq \{(x,j,w) \in \Delta : \\
&\hspace{2cm}\|A_j x\|_{P_{\star}(\omega_N,\delta)} > \overline{\gamma}(\omega_N,\delta) \|x\|_{P_{\star}(\omega_N,\delta)} \},
\end{align*}
with
\[
\overline{\gamma}(\omega_N,\delta) \coloneqq \gamma_{\star}(\omega_N,\delta) + W\sqrt{\kappacond(P_{\star}(\omega_N,\delta))},
\]
and $\kappacond(P)$ is defined in Theorem~\ref{thm:indirect-jsr-bound}.
\end{lem}
\begin{pf}
In this proof, with a slight abuse of notation, we denote $\gamma=\gamma_\star(\omega_N,\delta)$, $P=P_\star(\omega_N,\delta)$ and $\overline\gamma=\overline\gamma(\omega_N,\delta)$.
Also, for each sample set $\omega_N$, we define
\[
\tilde{V}(\omega_N) \coloneqq \{ (x,j,w) : \|A_jx+w\|_P > \gamma\|x\|_P \}.
\]
Following \cite[Eqs. 18--19]{banse2024learning}, it holds that
\[
\pi^N\left(\left\{\omega_N\in\Delta^N : \pi(\tilde{V}(\omega_N)) \leq \epsilon(\beta,N)\right\}\right) \geq 1-\beta .
\]
It remains to relate $\pi(\tilde V(\omega_N))$ and $\pi(V(\omega_N))$ for each $\omega_N$.
For that, fix $\omega_N\in\Delta^N$ and let $(x,j,\tilde{w})\in V(\omega_N)$.
We will prove that for all
$w\in W$,
\begin{equation}\label{eq:direct-dual}
\!\!(\:\underbrace{\|w\|_2 \leq W}_{(a)} \:\vee\: \underbrace{(A_jx)^\top Pw \geq 0}_{(b)}\:) \Rightarrow (x,j,w)\in\tilde{V}(\omega_N).
\end{equation}
Indeed, let $w\in\mathcal{W}$.
In case (a), by the triangle inequality, we have
\begin{align*}
    \overline{\gamma} \|x\|_P < \|A_jx\|_P &\leq \|A_jx+w\|_P + \|w\|_P \\
    &\leq\|A_jx+w\|_P + W\sqrt{\lambda_{\text{max}}(P)}
\end{align*}
and
\begin{align*}
\overline{\gamma} \|x\|_P &= \gamma \|x\|_P + W \sqrt{\kappacond(P)}\|x\|_P \\
    &\geq \gamma \|x\|_P + W \sqrt{\lambda_{\text{max}}(P)}.
\end{align*}
Therefore, we have
\begin{equation*}
    \gamma \|x\|_P \leq \overline{\gamma} \|x\|_P - W \sqrt{\lambda_{\text{max}}(P)} < \|A_jx+w\|_P,
\end{equation*}
which implies $(x,j,w)\in\tilde V(\omega_N)$.
Now, in case (b), we have
\begin{align*}
    &(A_jx+w)^\top P(A_jx+w) \\
    &\qquad= (A_jx)^\top P (A_jx) + 2\underbrace{(A_jx)^\top P w}_{\geq 0} + \underbrace{w^TPw}_{\geq 0} \\
    &\qquad\geq (A_jx)^\top P (A_jx).
\end{align*}
Hence,
\[
\|A_jx+w\|_P\ge\|A_jx\|_P>\overline{\gamma}\|x\|_P\geq\gamma\|x\|_P,
\]
which implies $(x,j,w)\in\tilde V(\omega_N)$.
This proves \eqref{eq:direct-dual}.

Then, by \eqref{eq:direct-dual}, Assumptions \ref{assum:noise} and \ref{assum:zero-measure}, and the definitions of $\tilde{V}(\omega_N)$ and $V(\omega_N)$, it follows that (letting $\lVert\cdot\rVert=\lVert\cdot\rVert_P$)
\begin{align*}
\pi(\tilde{V}(\omega_N)) &=\mathbb{E}_{x,j,w}\big[ \boldsymbol{1}_{\|A_jx+w\|>\gamma\|x\|}\big] \\
&=\frac12\mathbb{E}_{x,j,w}\big[ \boldsymbol{1}_{\|A_jx+w\|>\gamma\|x\|}+\boldsymbol{1}_{\|A_jx-w\|>\gamma\|x\|}\big] \\
\intertext{(since $\pi$ is symmetric with respect to $x$)}
&\geq\frac12\mathbb{E}_{x,j,w}\big[ \boldsymbol{1}_{\|A_jx\|>\overline\gamma\|x\|}\boldsymbol{1}_{(A_jx)^\top Pw\ge0} \\
&\hspace{0.5cm} {}+\boldsymbol{1}_{\|A_jx\|>\overline\gamma\|x\|}\boldsymbol{1}_{(A_jx)^\top Pw<0}\boldsymbol{1}_{\lVert w\rVert_2\leq W} \\
&\hspace{0.5cm} {}+\boldsymbol{1}_{\|A_jx\|>\overline\gamma\|x\|}\boldsymbol{1}_{(A_jx)^\top Pw\le0} \\
&\hspace{0.5cm} {}+\boldsymbol{1}_{\|A_jx\|>\overline\gamma\|x\|}\boldsymbol{1}_{(A_jx)^\top Pw>0}\boldsymbol{1}_{\lVert w\rVert_2\leq W}\big] \\
\intertext{(by the previous results)}
%&=\frac12\mathbb{E}_{x,j,w}\big[ \boldsymbol{1}_{\|A_jx\|>\overline\gamma\|x\|} \\
%&\hspace{0.5cm}{}\cdot\big(\big\{\boldsymbol{1}_{(A_jx)^\top Pw\ge0}+\boldsymbol{1}_{(A_jx)^\top Pw\le0} \big\} \\
%&\hspace{0.5cm}{}+\big\{\boldsymbol{1}_{(A_jx)^\top Pw<0}+\boldsymbol{1}_{(A_jx)^\top Pw>0} \big\} \boldsymbol{1}_{\lVert w\rVert_2\leq W}\big)\big]\\
&=\frac12\mathbb{E}_{x,j}\big[ \boldsymbol{1}_{\|A_jx\|>\overline\gamma\|x\|} \\
&\hspace{0.5cm}{}\cdot\mathbb{E}_w\big(\big\{\boldsymbol{1}_{(A_jx)^\top Pw\ge0}+\boldsymbol{1}_{(A_jx)^\top Pw\le0} \big\} \\
&\hspace{0.5cm}{}+\big\{\boldsymbol{1}_{(A_jx)^\top Pw<0}+\boldsymbol{1}_{(A_jx)^\top Pw>0} \big\} \boldsymbol{1}_{\lVert w\rVert_2\leq W}\big)\big]\\
\intertext{(by the independence of the distributions)}
&\geq\frac12\mathbb{E}_{x,j}\big[ \boldsymbol{1}_{\|A_jx\|>\overline\gamma\|x\|}(1+1-p(W))\big] \\
\intertext{(using Assumptions \ref{assum:noise} and \ref{assum:zero-measure})}
&= \Big(1-\frac{1}{2}p(W)\Big) \pi(V(\omega_N)).
\end{align*}
Hence,
\[
\pi(V(\omega_N)) \leq \frac{\pi(\tilde{V}(\omega_N))}{1-\frac12p(W)},
\]
which concludes the proof.\qed
\end{pf}
{\renewcommand{\Elproofname}{Proof of Theorem \ref{thm:jsr-direct-bound}.}
\begin{pf}
In this proof, with a slight abuse of notation, we denote $P=P_\star(\omega_N,\delta)$ and $\overline\gamma=\overline\gamma(\omega_N,\delta)$.
For each sample set $\omega_N$ and $j\in[m]$, we let
\[
\calV_j(\omega_N) \coloneqq \left\{x\in\mathbb S^{n-1} : \|A_jx\|_P>\overline\gamma\|x\|_P \right\}.
\]
Let $V(\omega_N)$ be as in Lemma \ref{lem:direct-violation-set_bound}.
By Assumption~\ref{assump-mode-sampling} and with the same reasoning as in the proof of Theorem \ref{thm:indirect-jsr-bound}, it holds that for all $j\in[m]$,
\[
\mustate(\calV_j(\omega_N))\le\frac{1}{\alpha}\pi(V(\omega_N)).
\]
Lemma \ref{lem:direct-violation-set_bound} implies that
\[
\pi^N\!\left(
\left\{\omega_N \in \Delta^N :
\mustate(\calV_j(\omega_N)) \leq \frac1\alpha\xi(\beta,N)
\right\}
\right)
\geq 1-\beta.
\]
Hence, the proof is completed by applying Proposition \ref{prop:kenenian} with $\gamma=\overline\gamma$ and $\nu=\frac1\alpha\xi(\beta,N)$.\qed
\end{pf}}

\subsection{Overall direct framework}
The overall framework is given in Algorithm~\ref{alg:direct}.
The soundness of the algorithm is established below:
\begin{thm}
Consider a SLS $\mathcal{A}$. Let $\beta\in(0,1)$.
Let $\omega_N$ be an i.i.d.~sample set as in \eqref{eq:sample_set_equation} and $\calD$ be the observation set as in \eqref{eq:observation-set}. Let Assumptions \ref{assump-mode-sampling} and \ref{assum:noise} hold.
Consider an execution of Algorithm \ref{alg:direct} with $\calD$ and $\beta$ as inputs and $\overline{\rho}$ as output.
With confidence $1-\beta$ on the sampling of $\omega_N$, it holds that $\rho(\mathcal{A}) \leq \overline{\rho}$.
\end{thm}
\begin{pf}
Corollary of Theorem \ref{thm:jsr-direct-bound}.\qed
\end{pf}
\begin{algorithm}[t]
\caption{Direct data-driven stability certification}
\label{alg:direct}
\begin{algorithmic}[1]
\Require Dataset $\calD\coloneqq\{(x_i,y_i)\}_{i=1}^N$ of noisy trajectories from unknown SLS $\calA$; confidence level $\beta\in(0,1)$.

\Ensure Probabilistic upper bound $\overline{\rho}$ on $\rho(\calA)$.

\State \textbf{Quadratic Synthesis:}  
Compute a matrix $P_\star(\calD,\delta)$ and associated scenario contraction rate $\gamma_\star(\calD,\delta)$ as solutions of \eqref{prob:data-driven_problem}.

\State \textbf{Bound on the JSR of the system:}  
Compute $\overline{\rho} = \overline{\rho}(\mathcal{D})$ as in Theorem~\ref{thm:jsr-direct-bound}.
\end{algorithmic}
\end{algorithm}

\section{Experimental Results}\label{sec:exp}

%\subsection{Consensus network}\label{sec:consensus}

We consider a consensus problem over a hidden switched interaction network, illustrated in Figure~\ref{fig:network}. Determining whether consensus is achieved can be formulated as a stability analysis problem for a $5$D SLS \cite{jadbabaie2003coordination}. Since the underlying networks are assumed unknown, stability is assessed by using data-driven methods on trajectory data.
We apply both the direct and indirect approaches to this problem and compare their performance under different noise settings.
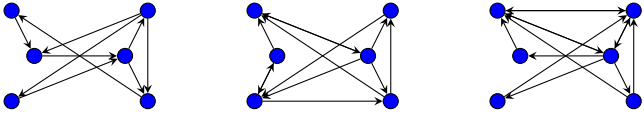
\begin{figure}
    \centering
    \begin{tikzpicture}[scale=0.6]
        % Graph 1
        \node[draw, circle, fill=blue, inner sep=2pt] (1) at (0,2) {};
        \node[draw, circle, fill=blue, inner sep=2pt] (2) at (0,0) {};
        \node[draw, circle, fill=blue, inner sep=2pt] (3) at (0.5,1) {};
        \node[draw, circle, fill=blue, inner sep=2pt] (4) at (3,0) {};
        \node[draw, circle, fill=blue, inner sep=2pt] (5) at (3,2) {};
        \node[draw, circle, fill=blue, inner sep=2pt] (6) at (2.5,1) {};
        \draw[->,>=stealth] (4) -- (1);
        \draw[->,>=stealth] (5) -- (2);
        \draw[->,>=stealth] (1) -- (3);
        \draw[->,>=stealth] (5) -- (3);
        \draw[->,>=stealth] (5) -- (4);
        \draw[->,>=stealth] (6) -- (4);
        \draw[->,>=stealth] (6) -- (5);
        \draw[->,>=stealth] (2) -- (6);
        \draw[->,>=stealth] (3) -- (6);
    \end{tikzpicture}
    \hfill
    \begin{tikzpicture}[scale=0.6]
        % Graph 2
        \node[draw, circle, fill=blue, inner sep=2pt] (1) at (0,2) {};
        \node[draw, circle, fill=blue, inner sep=2pt] (2) at (0,0) {};
        \node[draw, circle, fill=blue, inner sep=2pt] (3) at (0.5,1) {};
        \node[draw, circle, fill=blue, inner sep=2pt] (4) at (3,0) {};
        \node[draw, circle, fill=blue, inner sep=2pt] (5) at (3,2) {};
        \node[draw, circle, fill=blue, inner sep=2pt] (6) at (2.5,1) {};
        \draw[->,>=stealth] (3) -- (1);
        \draw[->,>=stealth] (4) -- (1);
        \draw[->,>=stealth] (6) -- (1);
        \draw[->,>=stealth] (3) -- (2);
        \draw[->,>=stealth] (5) -- (2);
        \draw[->,>=stealth] (6) -- (2);
        \draw[->,>=stealth] (2) -- (3);
        \draw[->,>=stealth] (2) -- (4);
        \draw[->,>=stealth] (6) -- (4);
        \draw[->,>=stealth] (4) -- (5);
        \draw[->,>=stealth] (6) -- (5);
        \draw[->,>=stealth] (1) -- (6);
    \end{tikzpicture}
    \hfill
    \begin{tikzpicture}[scale=0.6]
        % Graph 3
        \node[draw, circle, fill=blue, inner sep=2pt] (1) at (0,2) {};
        \node[draw, circle, fill=blue, inner sep=2pt] (2) at (0,0) {};
        \node[draw, circle, fill=blue, inner sep=2pt] (3) at (0.5,1) {};
        \node[draw, circle, fill=blue, inner sep=2pt] (4) at (3,0) {};
        \node[draw, circle, fill=blue, inner sep=2pt] (5) at (3,2) {};
        \node[draw, circle, fill=blue, inner sep=2pt] (6) at (2.5,1) {};
        \draw[->,>=stealth] (3) -- (1);
        \draw[->,>=stealth] (4) -- (1);
        \draw[->,>=stealth] (5) -- (1);
        \draw[->,>=stealth] (6) -- (1);
        \draw[->,>=stealth] (6) -- (2);
        \draw[->,>=stealth] (6) -- (3);
        \draw[->,>=stealth] (6) -- (4);
        \draw[->,>=stealth] (1) -- (5);
        \draw[->,>=stealth] (2) -- (5);
        \draw[->,>=stealth] (4) -- (5);
        \draw[->,>=stealth] (6) -- (5);
        \draw[->,>=stealth] (1) -- (6);
        \draw[->,>=stealth] (5) -- (6);
    \end{tikzpicture}
    \caption{Interaction networks used in the consensus experiment.
    The system switches between three unknown network topologies, giving an underlying switched linear dynamics.}
    \label{fig:network}
\end{figure}
The mode selection is considered uniform, i.e., each of the three modes can be selected with probability $\frac{1}{3} = \alpha$.
We consider three noise settings:
\begin{itemize}
    \item \textbf{Gaussian noise:} additive white Gaussian noise with zero mean and covariance $\sigma^2 I$, where $\sigma = 10^{-3}$.
    \item \textbf{Gaussian mixture noise:} 
    \begin{align*}
    w_t \sim (1-\epsilon)\mathcal{N}(0,\sigma_1^2 I) + \epsilon \mathcal{N}(0,\sigma_2^2 I),
    \end{align*}
    where $\sigma_1 = 10^{-3}$, $\sigma_2 = 1$, and $\epsilon = \frac{1}{250}$. This model captures rare but large outliers.
    
    \item \textbf{Bounded noise:} noise uniformly distributed in the Euclidean ball $B(0,0.1)$. This setting enables comparison with guarantees derived under bounded noise assumptions, such as in \cite{banse2024learning}.
\end{itemize}
The quality of the indirect approach depends on the system identification procedure. Identifying SLS is known to be NP-hard in general, but heuristics often yield satisfactory practical results; see \cite{lauer2019hybrid} for an overview.
In this work, we employ the $k$-LinReg algorithm \cite[Algorithm~6]{lauer2019hybrid}, a variant of $k$-means for switched linear regression that alternates between assigning data points to modes and least-squares regression to estimate the matrices. To improve robustness to outliers, we adopt a trimmed variant inspired by trimmed $k$-means \cite{garcia1999robustness}, discarding a fraction of samples with the largest residuals at each iteration. The resulting procedure is summarized in Algorithm \ref{algo:trimmed-k-linReg} in Appendix \ref{appendix:algo}.
This highlights a key advantage of the indirect approach: the framework is agnostic to the identification method and can incorporate robust or domain-specific algorithms.

\subsection{Gaussian noise}\label{ssec:gauss_noise}
We set $\delta = 5$. For the probabilistic bounds, we use $\beta = 5\%$, $W = 3\sigma = 4 \cdot 10^{-3}$ (yielding $p(W) \approx 0.006$), and $\epsilon = 5\cdot 10^{-3}$.\footnote{We chose the parameters $W$ and $\epsilon$ to minimize the upper bound; see Figures \ref{fig:consensus_ub_gaussian_indirect_epsilon} and \ref{fig:consensus_ub_gaussian_direct}} in Appendix \ref{apendix:add_fig}. For the indirect approach based on binomial-tail inversion bounds, the dataset is split equally into identification and certification sets.

%Figure~\ref{fig:consensus_gamma_gaussian} shows the estimated quadratic joint spectral radius (JSR) as a function of the number of samples $N$. The indirect method achieves accurate estimates with relatively few samples, whereas the direct approach requires significantly larger datasets and does not attain comparable accuracy.

Figure~\ref{fig:consensus_ub_gaussian_direct_vs_indirect} reports the bounds obtained by both approaches as a function of the total number of samples.
Stability is certified with 95\% confidence when the upper bound drops below $1$. The indirect method certifies stability with considerably fewer samples (around $800$) compared to the direct approach (around $2200$).

Figure~\ref{fig:consensus_time} compares the computational cost of the direct and indirect approaches with the number of samples. The direct method requires solving an LMI, whereas the indirect method primarily incurs the cost of the $k$-LinReg identification procedure. Since solving LMIs is significantly more expensive than alternating least squares, the indirect approach is substantially faster.

Overall, the indirect approach performs better than the direct one in both sample and computational complexity, consistent qualitatively with \cite[Subsection II.C]{krishnan2021direct}, which reports similar advantages for indirect methods in data-driven predictive control of LTI systems.
\begin{figure}
    \centering
    \includegraphics[width=8.6cm]{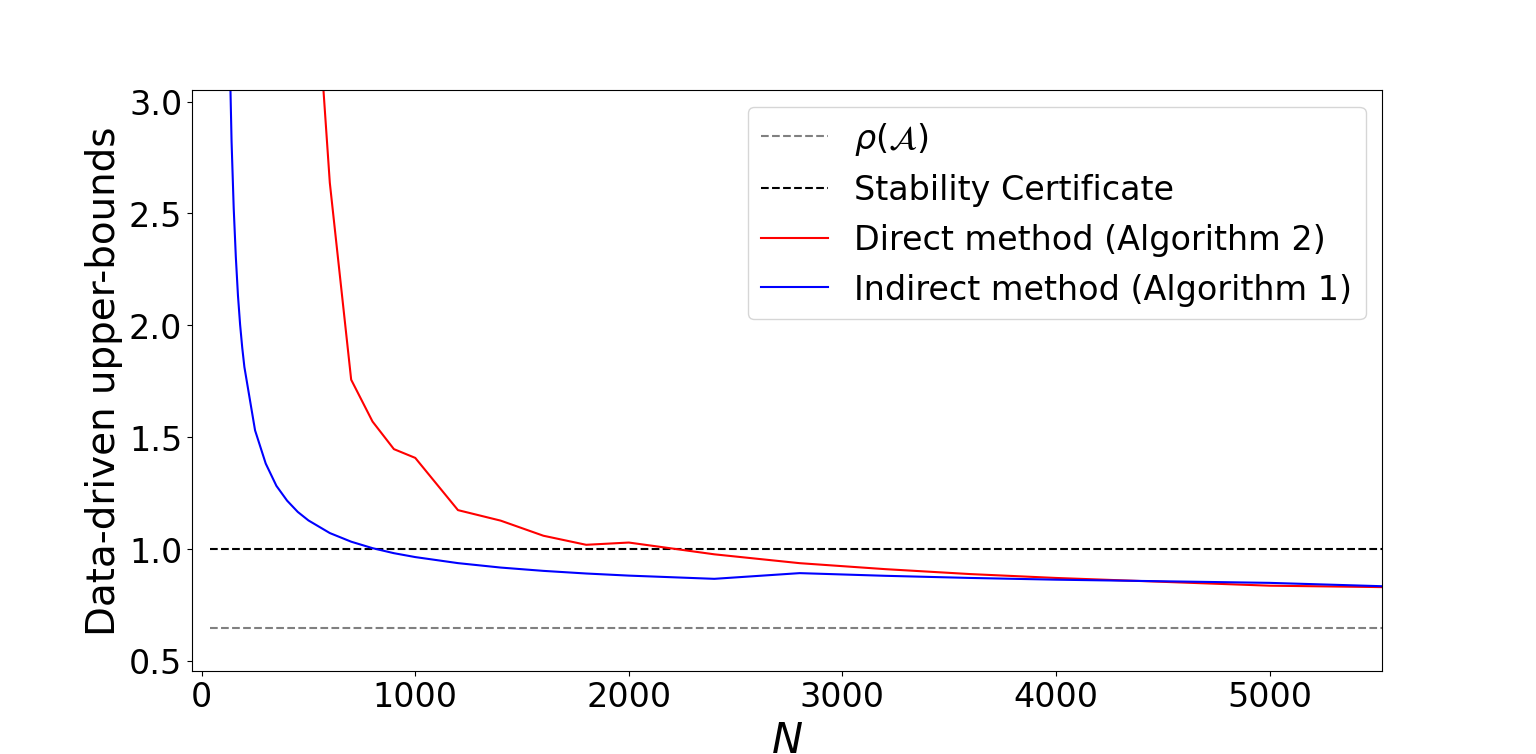}
    \caption{\textbf{Comparison of the direct and indirect probabilistic upper bounds on the consensus problem with Gaussian noise.}
    Stability is certified (with confidence 95\%) when the bound drops below 1.
    The indirect approach certifies stability with fewer samples than the direct method.}
    \label{fig:consensus_ub_gaussian_direct_vs_indirect}
\end{figure}
\begin{figure}
    \centering
    \includegraphics[width=8.6cm]{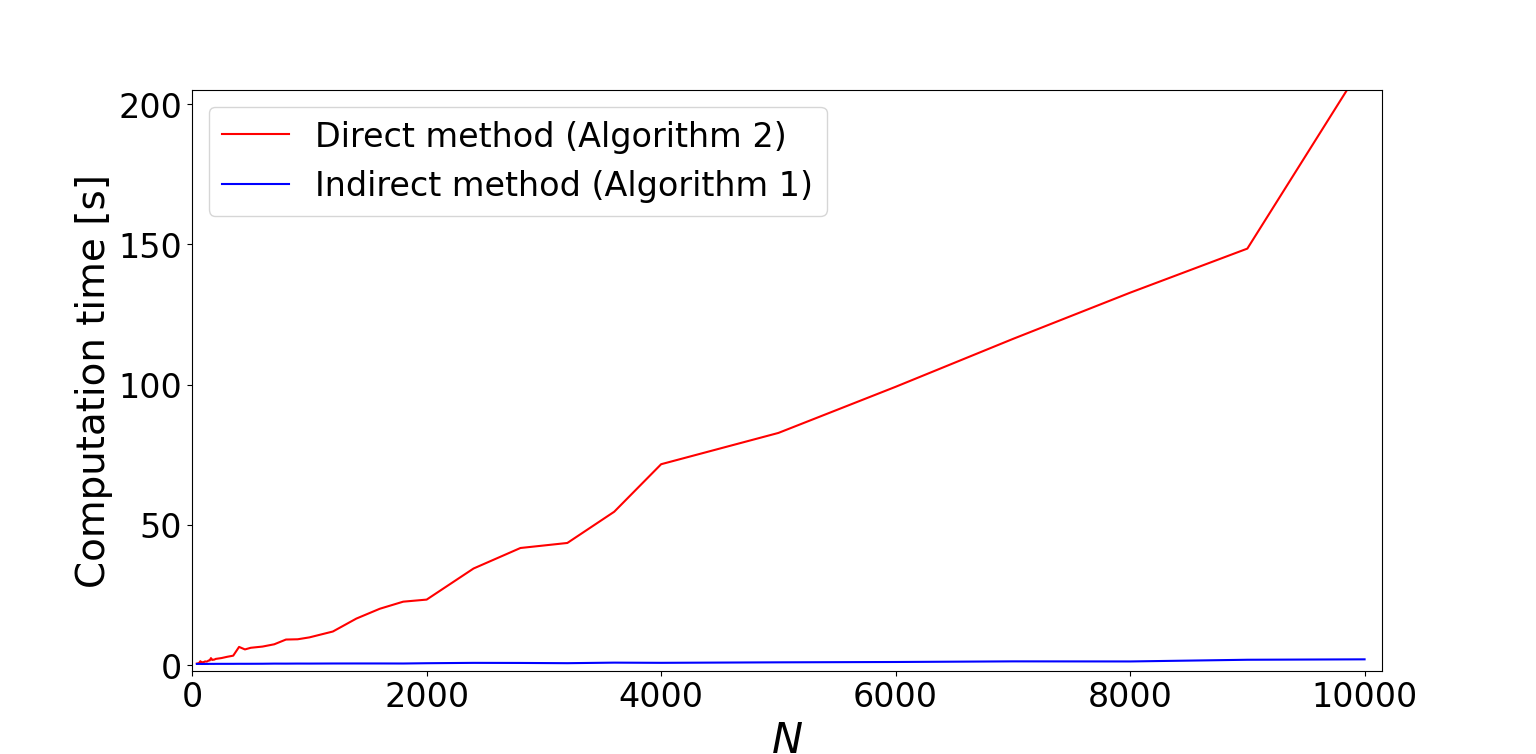}
    \caption{\textbf{Computation time of the direct method (Algorithm \ref{alg:direct}) and the indirect method (Algorithm \ref{alg:indirect}).}
The direct method becomes increasingly more expensive due to the growth in the number of LMI constraints, whereas the indirect approach remains computationally efficient.
    \label{fig:consensus_time}}
\end{figure}

\subsection{Gaussian mixture noise}\label{ssec:gauss_mixt_noise}
We now consider the Gaussian mixture noise setting with outliers. We use $\delta = 5$, $\beta = 5\%$, $\epsilon = 5 \cdot 10^{-3}$ and $W = 4 \cdot 10^{-3}$, leading to $p(W) \approx 0.006 + \frac{1}{250}$.

Figure~\ref{fig:consensus_ub_gaussian_direct_vs_indirect_mixture} reports the bounds.
The indirect method provides meaningful stability guarantees, while the direct method fails to produce useful bounds in this setting.
These results highlight the robustness advantage of the indirect paradigm when outliers are present. Since the direct approach enforces a decrease of the Lyapunov function along all sampled trajectories, it is inherently sensitive to even a small fraction of corrupted data. By contrast, in the indirect approach, the identification step implemented by Algorithm \ref{algo:trimmed-k-linReg} acts as a kind of filter for the data that removes the outliers \cite{lei2012automatic}, thereby mitigating these effects when sufficient data are available.
This limitation reflects the sensitivity of the current scenario-based formulation to outliers. Incorporating robust scenario optimization techniques, such as those proposed in \cite{romao2021tight}, constitutes a promising future direction.
\begin{figure}
    \centering
    \includegraphics[width=8.6cm]{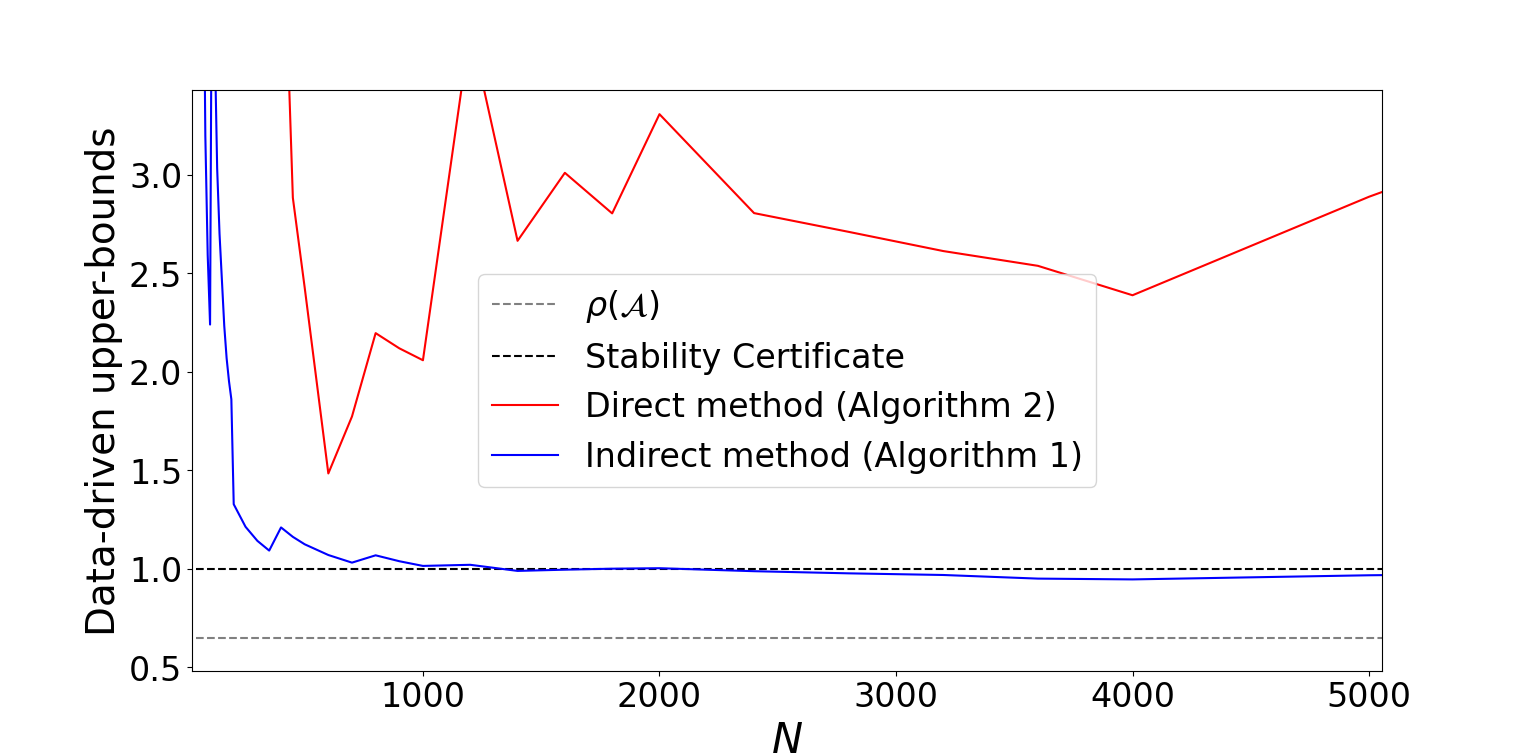}
    \caption{\textbf{Comparison under Gaussian mixture noise.}
    The indirect approach provides meaningful stability guarantees (below $1$), whereas the direct approach fails to yield informative bounds due to its sensitivity to outliers.
    \label{fig:consensus_ub_gaussian_direct_vs_indirect_mixture}}
\end{figure}

\subsection{Bounded noise}\label{subsec:bounded_noise_exp}

We consider the bounded noise setting to compare the direct approach of \cite{banse2024learning}, which assumes bounded noise, with the bound obtained in Theorem \ref{thm:jsr-direct-bound}.
The bound from \cite{banse2024learning} corresponds to the special case of Theorem \ref{thm:jsr-direct-bound} with $W = 0.1$ and $p(W)=0$. We compare with the case $W=0$ and $p(W) = 1$.
We use $\delta = 5$ and $\beta=5\%$.

Figure~\ref{fig:consensus_ub_bounded_direct} reports the results.
Our bound quickly falls below that of \cite{banse2024learning}.
In particular, the bound from \cite{banse2024learning} fails to certify stability, whereas our bound successfully drops below one around $7500$ samples. This gap is due to the worst-case noise analysis in \cite[Proposition 3, Theorem 1]{banse2024learning}, which introduces the additive term $W\sqrt{\kappa_{\mathrm{cond}}(P)}$. This term does not vanish with the number of samples, leading to inherent asymptotic conservatism in the bounded-noise setting.
In contrast, Theorem \ref{thm:jsr-direct-bound} is more flexible and can avoid this situation by setting the additional term to $0$.
More generally, one can optimize over $W$ in Theorem~\ref{thm:jsr-direct-bound} to obtain the tightest bound, as is illustrated in Appendix~\ref{apendix:add_fig}.
\begin{figure}
    \centering
    \includegraphics[width=8.6cm]{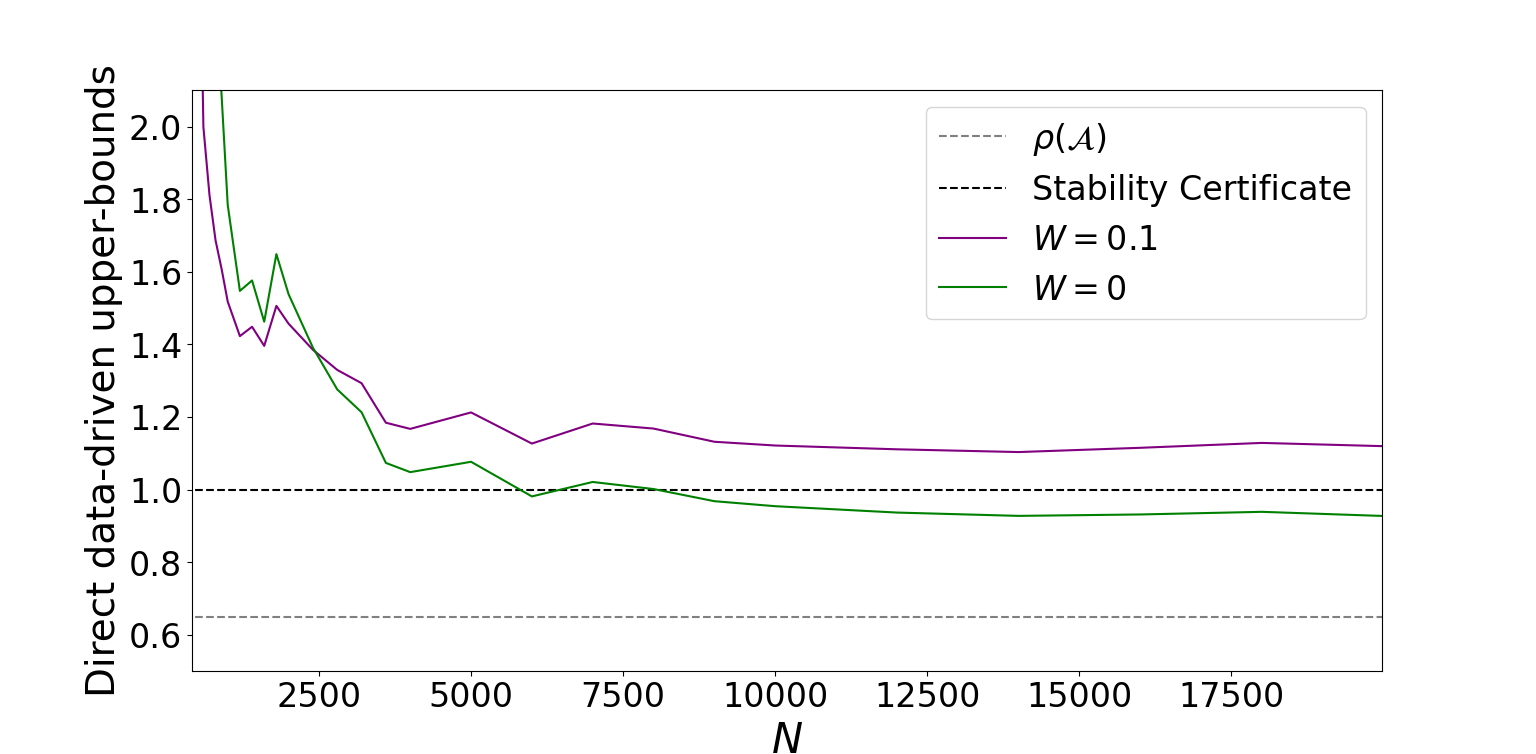}
    \caption{\textbf{Comparison under bounded noise between \cite[Theorem~1]{banse2024learning} ($W=0.1$) and Theorem \ref{thm:jsr-direct-bound} ($W=0$)}. 
Our bound converges to a smaller value and certifies stability, while the bound in \cite{banse2024learning} does not.}
    \label{fig:consensus_ub_bounded_direct}
\end{figure}

\section{Conclusion}\label{sec:conclusion}
In this paper, we developed a new framework for indirect data-driven stability analysis of switched linear systems by combining statistical guarantees for system identification with Lyapunov-based stability certification. We also improved the finite-sample bounds for direct data-driven stability analysis from \cite{banse2024learning}.

The resulting direct and indirect methods were compared experimentally. Our results show that the indirect approach can significantly outperform the direct approach, especially when the dataset is small or contains outliers. In particular, it can provide nontrivial stability certificates in situations where the direct approach fails to do so, even asymptotically.

Beyond the specific methods developed in this paper, these results contribute to the broader discussion on the relative merits of direct and indirect approaches. In line with the observations of \cite{krishnan2021direct}, the indirect approach often exhibits superior performance in limited-data regimes. At the same time, our findings indicate significant practical differences between direct and indirect methods for switched system stability certification, in sharp contrast to recent linear-control studies highlighting close connections between the two paradigms \cite{dorfler2023bridgingdirectvsindirect,dorfler2023certaintyequivalence}.

Understanding the origin of these differences and developing direct methods that are more robust to outliers constitute promising directions for future research. We also plan to extend the comparison to adaptive sampling strategies developed for the direct method, which appear promising for improving sample efficiency \cite{vuille2024data,vuille2025stochastic}.

\bibliographystyle{plain}        % Include this if you use bibtex 

\bibliography{autosam}           % and a bib file to produce the 

\appendix

\section{Examples for Assumption~\ref{assum:noise}}\label{appendix:exmp}

\begin{exmp}
Assume the noise $w$ has zero mean and covariance matrix $R$.
Then by a direct application of Markov's inequality to $w^{\top} w = \|w\|_2^2$,
\begin{equation*}
\munoise(\{w \in \mathbb{R}^n:\|w\|_2>W) \leq \frac{\mathrm{tr}(R)}{W^2}.
\end{equation*}
\end{exmp}
\begin{exmp}
Assume that the noise is a scaled standard Gaussian noise in dimension $n$, i.e., $w \sim N(0,\sigma^2I)$, with $\sigma>0$.
Then $\frac{\|w\|_2^2}{\sigma^2} \sim \chi_n^2$, where $\chi_n^2$ is the standard Chi-square distribution with $n$ degrees of freedom, and the tail probability can be computed explicitly as
\begin{equation*}
    \munoise(\|w\|_2 > W)
 =
 1 - \frac{\gamma\!\left(\frac{n}{2}, \frac{W^2}{2\sigma^2}\right)}{\Gamma\!\left(\frac{n}{2}\right)},
\end{equation*}
where $\Gamma(\cdot)$ is the Gamma function and $\gamma(\cdot,\cdot)$ is the lower incomplete Gamma function.
\end{exmp}

\section{Proof of Proposition \ref{prop:kenenian}}\label{appendix:proof_kenenian}
%Given $j \in [m]$, let $V_j \coloneqq \{x \in \mathbb{S}^{n-1}:\frac{\|A_jx\|_P}{\|x\|_P} > \overline{\gamma}\}$, then by assumption of the theorem we have $\mustate(V_j) \leq \epsilon$ with probability $1-\beta$ on the sampling of $\omega_N$.\\
Let $P=L^\top L$ be the Cholesky factorization of $P$ and define $B_j \coloneqq LA_jL^{-1}$. % ($\Leftrightarrow A_j = L^{-1} B_j L$).
Given $j \in [m]$,
\begin{align*}
    &\quad x^\top A_j^\top P A_j x \leq \overline{\gamma}^2 x^\top P x,\;\;\forall x \in (\mathbb{S}^{n-1} \setminus V_j)\\
    %&\Leftrightarrowx^\top L^\top B_j^\top L^{-T} (L^\top L) L^{-1} B_j L x \leq \overline{\gamma}^2 x^\top L^\top L x,\;\;\forall x \in (\mathbb{S}^{n-1} \setminus V_j)\\
    %&\Leftrightarrow(Lx)^\top B_j^\top B_j (L x) \leq \overline{\gamma}^2 (Lx)^\top (L x),\;\;\forall x \in (\mathbb{S}^{n-1} \setminus V_j)\\
    &\Leftrightarrow x^\top B_j^\top B_j x \leq \overline{\gamma}^2 x^\top x,\;\;\forall x \in L(\mathbb{S}^{n-1} \setminus V_j)\\
    &\Leftrightarrow\|B_j x\|_2 \leq \overline{\gamma} \|x\|_2,\;\;\forall x \in \Pi_{\mathbb{S}^{n-1}}(L(\mathbb{S}^{n-1} \setminus V_j))\\
    &\Leftrightarrow\|B_j x\|_2 \leq \overline{\gamma},\;\;\forall x \in (\mathbb{S}^{n-1}\setminus \Pi_{\mathbb{S}^{n-1}}L(V_j)),
\end{align*}
where $\Pi_{\mathbb{S}^{n-1}}$ is the scaling on $\mathbb{S}^{n-1}$ and the third inequality holds by homogeneity of the system.
Denote $V'_j\coloneqq\Pi_{\mathbb{S}^{n-1}}(L(V_j))$. By \cite[Theorem 15]{kenanian2019data}, we have: $\mustate(V'_j) \leq \kappa(P) \mustate(V_j)$.
Now, let $\xi_j > 0$ be the radius of the largest sphere included in $\mathrm{conv}(\mathbb{S}^{n-1}\setminus V'_j)$. Then, given $z \in \mathbb{S}^{n-1}\setminus V'_j$, there exists $x_1,x_2 \in (\mathbb{S}^{n-1}\setminus V'_j)$ and $\lambda \in [0,1]$ such that $\xi_j z = \lambda x_1 + (1-\lambda) x_2$ and by the triangle inequality:
\begin{align*}
    &\quad\|B_j (\xi_j z)\|_2 \leq \lambda \|B_j x_1\|_2 + (1-\lambda) \|B_j x_2\|_2 \leq \overline{\gamma}\\
    &\Leftrightarrow\|B_j z\|_2 \leq \frac{\overline{\gamma}}{\xi_j}.
\end{align*}
Therefore, we have the following inequality:
\begin{align*}
    &\quad\|B_j x\|_2 \leq \frac{\overline{\gamma}}{\xi_j},\;\forall x \in \mathbb{S}^{n-1} \\
    &\Leftrightarrow\|B_j x\|_2 \leq \frac{\overline{\gamma}}{\xi_j} \|x\|_2,\;\forall x \in \mathbb{R}^{n} \\
    %&\Leftrightarrow\|LA_jL^{-1} x\|_2 \leq \frac{\overline{\gamma}}{\xi_j} \|x\|_2,\;\;\;\forall x \in \mathbb{R}^{n} \\
    %&\Leftrightarrow\|LA_jx\|_2 \leq \frac{\overline{\gamma}}{\xi_j} \|Lx\|_2,\;\;\;\forall x \in L^{-1}(\mathbb{R}^{n}) \\
    &\Leftrightarrow\|A_jx\|_P \leq \frac{\overline{\gamma}}{\xi_j} \|x\|_P,\;\forall x \in \mathbb{R}^n
\end{align*}
Finally, following \cite[Proposition 3]{kenanian2019data}, we have:
\begin{equation*}
    \xi_j \geq \zeta\Big(\frac{\mustate(V'_j)}2\Big) \geq \zeta\Big(\frac{\kappa(P)\mustate(V_j)}2\Big) \geq \zeta\Big(\frac{\eta\kappa(P)}2\Big).
\end{equation*}
Therefore, we have
\begin{equation*}
    \max_{x \in \mathbb{S}^{n-1}} \frac{\|A_jx\|_P}{\|x\|_P} \leq \frac{\overline{\gamma}}{\xi_j} \leq \frac{\overline{\gamma}}{\zeta\left(\frac12\kappa(P)\eta\right)}
\end{equation*}
and
\begin{equation*}
    \rho(\mathcal{A}) \leq \rhoquad(\mathcal{A},\delta) = \max_{\substack{x \in \mathbb{S}^{n-1}\\j \in [m]}} \frac{\|A_jx\|_P}{\|x\|_P} \leq \frac{\overline{\gamma}}{\zeta\left(\frac12\kappa(P)\eta\right)}.\qed
\end{equation*}
\section{System Identification Algorithm}\label{appendix:algo}

\begin{algorithm}[H]
\caption{Trimmed $k$-LinReg for SLS identification}
\label{algo:trimmed-k-linReg}
\begin{algorithmic}[1]

\Require Dataset $\calD\coloneqq\{(x_i,y_i)\}_{i=1}^N$ of noisy trajectories from unknown SLS $\calA$; number of modes $m$; trimming level $\theta\in(0,1)$; tolerance $\epsilon_{\mathrm{tol}}>0$; maximum number of iterations $T$.

%\State Randomly initialize $\hat{A}_j\in \mathbb{R}^{n \times n}$, $j = 1,\hdots,m$.
%\State Set $t \gets 0$.
\State Randomly initialize $\{\hat A_j\}_{j=1}^m$ and set $t\gets0$.
\Repeat

    \State $t \gets t+1$.
    
    \For{$i = 1$ to $N$}
        \State Let $\hat{r}_i \coloneqq \argmin\limits_{j \in \{1,\dots,m\}}
        \| y_i - \hat{A}_j x_i \|_2^2$.
    \EndFor

    \For{$j = 1$ to $m$}
    
        \State Let $I_j \coloneqq 
        \{i \in [N] : \hat{r}_i = j \}$.
        
        \If{$I_j \neq \emptyset$}

            \State Compute $e_i \coloneqq \| y_i - \hat{A}_j x_i \|_2^2$, $i \in I_j$.
            %\State Compute residuals $e_i \coloneqq \| y_i - \hat{A}_j x_i \|_2^2$, $i \in I_j$.
            
            \State Let $I_j^\theta$ be the indices of the $\lfloor(1-\theta)|I_j|\rfloor$ smallest residuals $e_i$ in $I_j$.
            %\State Let $I_j^{\theta} \subseteq I_j$ be the subset of $I_j$ containing the indices of the $\lfloor (1-\theta)|I_j| \rfloor$ smallest residuals among all residuals with indices in $I_j$.
            
            \State 
            Let $\hat{A}_j \coloneqq 
            \argmin\limits_{A \in \mathbb{R}^{n \times n}}
            \sum_{i \in I_j^{\theta}}
            \| y_i - A x_i \|_2^2$
            
        \EndIf
    \EndFor

\Until{$\max\limits_{j\in\{1,\ldots,m\}}
\|\hat A_j^{(t)}-\hat A_j^{(t-1)}\|_F
\le \epsilon_{\mathrm{tol}}
\;\:\text{or}\:\;
t \ge T$}

\State \Return $\{\hat{A}_j\}_{j=1}^m$
\end{algorithmic}
\end{algorithm}
%In all experiments, we set $\epsilon_{\mathrm{tol}} = 10^{-6}$ and $T = 100$. We used $\theta = 0.01$ in Secthe Gaussian-noise experiments and $\theta = 0.05$ in the Gaussian-mixture-noise experiments.
In all experiments, we set $\epsilon_{\mathrm{tol}} = 10^{-6}$, $T = 100$, $\theta = 0.01$ and $\theta = 0.05$ for Sections \ref{ssec:gauss_noise} and \ref{ssec:gauss_mixt_noise} respectively.
\section{Additional Figures and Comparisons}\label{apendix:add_fig}
Figures~\ref{fig:consensus_ub_gaussian_direct} and~\ref{fig:consensus_ub_gaussian_indirect_epsilon} present the guarantees obtained by the direct and indirect methods by varying several parameters.
Figure~\ref{fig:consensus_ub_gaussian_indirect_epsilon} compares the upper bound with the indirect method for different values of the parameter $\epsilon$.
Theorem \ref{thm:indirect-jsr-bound} using $\epsilon=5 \cdot 10^{-3}$ gives the tightest upper bound (the curve corresponding to $\epsilon = 3 \cdot 10^{-3}$ is not visible because the upper bound exceeds the plotted range.). Figure~\ref{fig:consensus_ub_gaussian_direct} compares the upper bound with the direct method for different values of the parameter $W$. Theorem \ref{thm:indirect-jsr-bound} with $W=4 \cdot 10^{-3}$ gives the tightest upper bound.

\begin{figure}
    \centering
    \includegraphics[width=8.6cm]{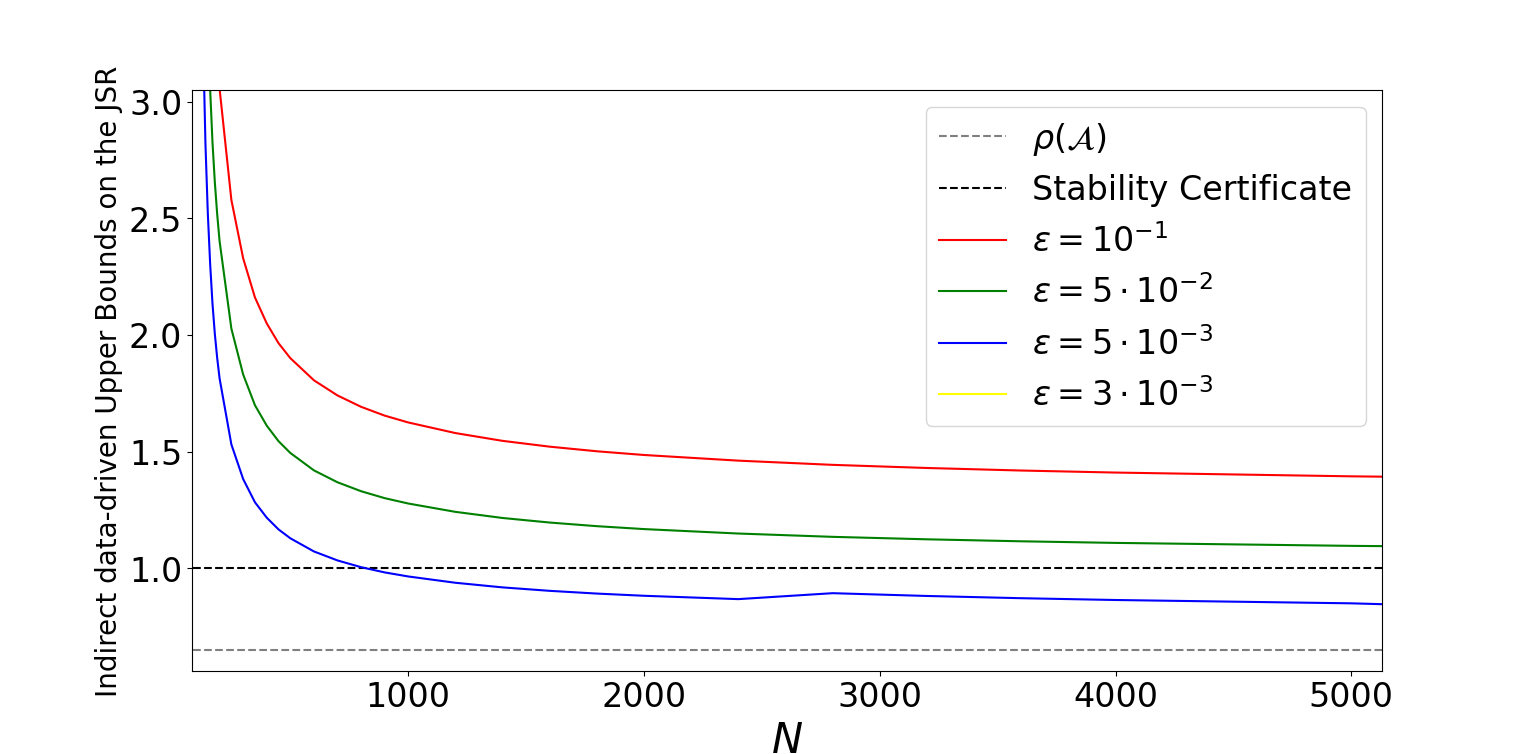}
    \caption{\textbf{Comparison of the probabilistic upper bounds with the indirect approach and different $\epsilon$}.
The value $\epsilon=5 \cdot 10^{-3}$ gives the tightest upper bound.}
    \label{fig:consensus_ub_gaussian_indirect_epsilon}
\end{figure}

\begin{figure}
    \centering
    \includegraphics[width=8.6cm]{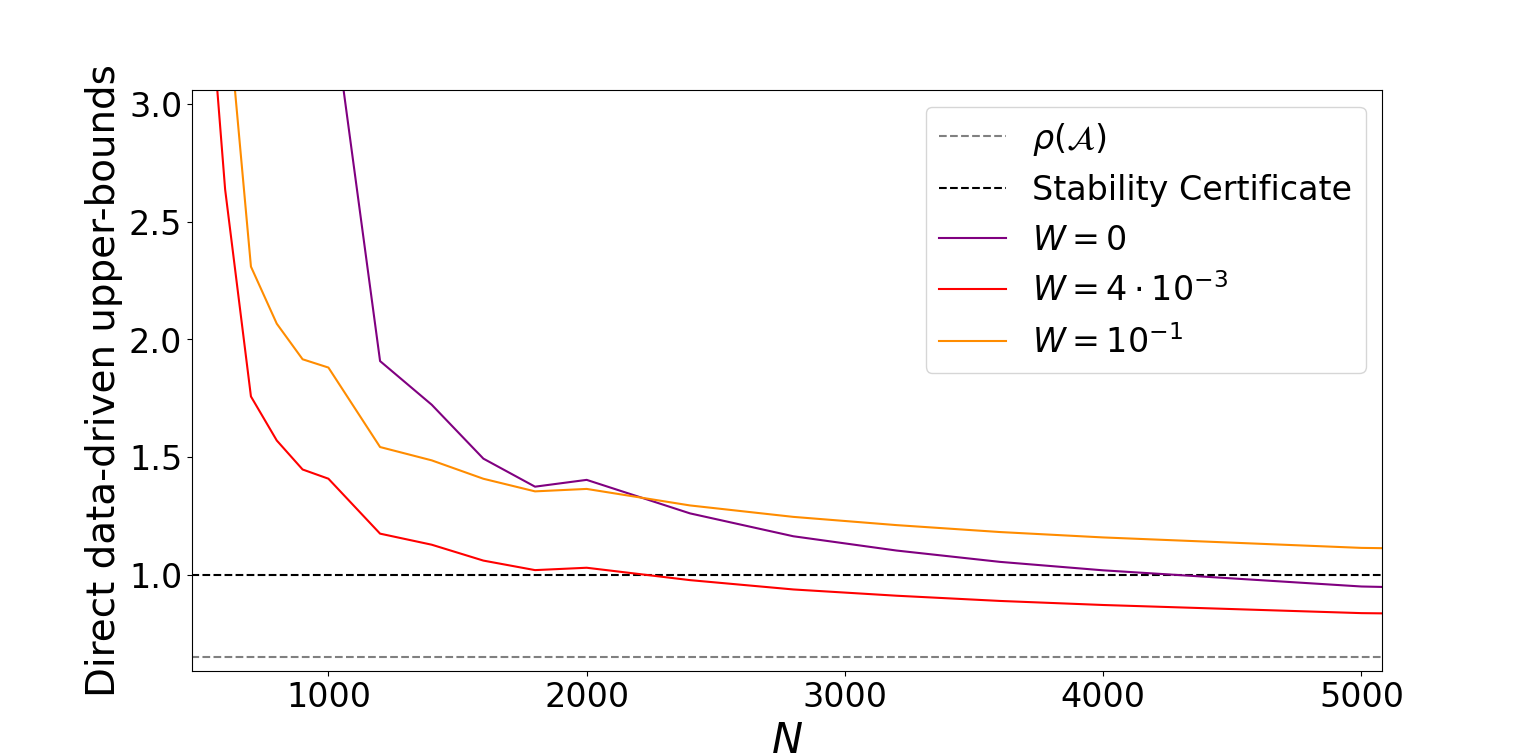}
    \caption{\textbf{Comparison of the probabilistic upper bounds obtained with the direct approach}.\\
$W=4 \cdot 10^{-3}$ gives the tightest upper bound.}
    \label{fig:consensus_ub_gaussian_direct}
\end{figure}

\end{document}